\documentclass[11pt]{amsart}

\usepackage{amscd}
\usepackage{xypic}  %commu
\usepackage{amssymb}
\usepackage{amsthm}
\usepackage{epsfig}

\usepackage{paralist}

\usepackage{tikz}

\usepackage[all,cmtip]{xy}
\usepackage{amsmath}
\usepackage{xcolor}

\usepackage[hidelinks=true,hypertexnames=false]{hyperref}

\newtheorem{thm}{Theorem}[section]

\newtheorem{lemma}[thm]{Lemma}
\newtheorem{prop}[thm]{Proposition}
\newtheorem{proposition}[thm]{Proposition}

\newtheorem{claim}[thm]{Claim}

\newtheorem{corollary}[thm]{Corollary}
\theoremstyle{definition}
\newtheorem{example}[thm]{Example}

\newtheorem{remark}[thm]{Remark}

  \newtheorem{definition-remark}[thm]{Definition-Remark}

\def\geq{\geqslant}
\def\leq{\leqslant}

\def\min{\operatorname{min}}

\def\c1{\operatorname{c_1}}
\def\c2{\operatorname{c_2}}

\def\Sym{\operatorname{Sym}}

\def\PP{{\mathbb P}}
\def\P{{\mathcal P}}
\def\A{{\mathcal A}}
\def\B{{\mathcal B}}
\def\D{{\mathcal D}}

\def\G{{\mathcal G}}
\def\L{{\mathcal L}}
\def\M{{\mathcal M}}

\def\O{{\mathcal O}}
\def\I{{\mathcal J}}

\def\H{{\mathcal H}}
\def\F{{\mathcal F}}

\def\V{{\mathcal V}}

\def\FF{{\mathbb F}}

\def\f{\mathfrak{f}}
\def\c{\mathfrak{c}}

\def\R{\mathcal{R}}

\def\x{\times}                   % product (fiber)
\def\cong{\simeq}

\def\+{\oplus}               % direct sum
\def\*{\otimes}                  % tensor product
\def\Aut{\operatorname{Aut}}

\def\Pic{\operatorname{Pic}}

\def\Num{\operatorname{Num}}

\begin{document}

\title[The locus of curves where the Prym--canonical map is not an embedding]{On the locus of Prym curves where the Prym--canonical map is not an embedding}

\author[C.~Ciliberto]{Ciro Ciliberto}
\address{Ciro Ciliberto, Dipartimento di Matematica, Universit\`a di Roma Tor Vergata, Via della Ricerca Scientifica, 00173 Roma, Italy}
\email{cilibert@mat.uniroma2.it}

\author[T.~Dedieu]{Thomas Dedieu}
\address{Thomas Dedieu,
Institut de Math\'ematiques de Toulouse~---~UMR5219,
Universit\'e de Toulouse~---~CNRS,
UPS IMT, F-31062 Toulouse Cedex 9, France} 
\email{thomas.dedieu@math.univ-toulouse.fr} 

\author[C.~Galati]{Concettina Galati}
\address{Concettina Galati, Dipartimento di Matematica e Informatica, Universit\`a della Calabria, via P. Bucci, cubo 31B, 87036 Arcavacata di Rende (CS), Italy}
\email{concettina.galati@unical.it}

\author[A.~L.~Knutsen]{Andreas Leopold Knutsen}
\address{Andreas Leopold Knutsen, Department of Mathematics, University of Bergen, Postboks 7800,
5020 Bergen, Norway}
\email{andreas.knutsen@math.uib.no}

%\keywords{Severi varieties, moduli map, nodal curves, $K3$ surfaces.}

%\subjclass{14H10, 14J28, 14B05.}

 \begin{abstract}  
 We prove that the locus of Prym curves $(C,\eta)$ of genus $g \geq 5$ for which  the Prym-canonical system  $|\omega_C(\eta)|$ is base point free but the Prym--canonical map is not an embedding is irreducible and unirational of dimension $2g+1$.
 \end{abstract}

\maketitle

\section{Introduction} \label{sec:intro}

Let  $g \geq 2$  and $\R_g$ be the moduli space of {\it Prym curves}, that is, of pairs $(C,\eta)$, with $C$ a smooth complex projective genus $g$ curve and $\eta$ a non--zero $2$--torsion point of $\Pic^0(C)$. It is well-known that $\R_g$ is irreducible of dimension $3g-3$ and that the natural forgetful map $\R_g \to \M_g$, where $\M_g$ denotes the moduli space of smooth genus $g$ curves, is finite of degree $2^{2g}-1$. 
The complete linear system $|\omega_C(\eta)|$ is of dimension $g-2$ and it is base point free unless $C$ is hyperelliptic and $\eta \cong \O_C(p-q)$, with $p$ and $q$ ramification points of the $g^1_2$ (cf.\ Lemma \ref{lemma:Prym2} below).

In this note we study the locally closed  locus  $\R^0_g$ in $\R_g$ of Prym curves $(C,\eta)$ such that  the {\it Prym-canonical system}  $|\omega_C(\eta)|$ is base point free but 
the morphism $C \to \PP^{g-2}$ it defines (the so-called {\it Prym--canonical map}) is  not an embedding.  Note that $\R^0_g$ is clearly dense in $\R_g$ for
$g \leq 4$. 
Our main result is the following:

\begin{thm} \label{thm:main} 
    Let $g \geq 5$.  The locus $\R^0_g$ is irreducible and unirational of dimension $2g+1$ and lies in the tetragonal locus.
\end{thm}

By the tetragonal locus $\R^1_{g,4}$ in $\R_g$ we mean  
the inverse image via $\R_g \to \M_g$ of the tetragonal locus
$\M^1_{g,4}$ of $\M_g$.

We also show:

\begin{prop} \label{prop:nod} 
  For general $(C,\eta) \in \R^0_g$, $g \geq 5$,   the Prym--canonical map is birational onto its image, and its image has precisely two nodes. 
\end{prop}

Although we believe that these results are of independent interest, our main motivation for studying the locus $\R_g^0$ is that it naturally contains pairs $(C,\eta)$ where $C$ is a smooth curve lying on an Enriques surface $S$ such that
\[ \phi(C)=\min\{E \cdot C \; | \: E \in \Pic (S), E>0, E^2=0\}=2,\]
and $\eta=\O_C(K_S)$, cf.\ Examples \ref{exa:enr1} and \ref{exa:enr2}
and Remark \ref{rem:enr}, in which case the Prym--canonical map
associated to $\eta$ is the restriction to $C$ of the map defined by
the complete linear system $|C|$ on $S$. The locus $\R_g^0$ indeed
naturally shows up in our recent work \cite{cdgk} concerning the
moduli of smooth curves lying on an Enriques surface, in which we use the
results in this note. Besides, we show in \cite{cdgk} that
$\R^0_g$ is dominated by curves on Enriques surfaces for $5 \leq g
\leq 8$. 

The paper is organized as follows. Section \ref {sec:proof} is devoted to 
recalling some  preliminary results. The irreducibility and unirationality of $\R^0_g$ is proved in \S  \ref {sec:irr}, whereas its dimension is computed in \S \ref {sec:dim}.  We conclude with the proof of Proposition \ref{prop:nod} together with the mentioned examples on Enriques surfaces.

\vspace{0.3cm} {\it Acknowledgements.} The authors thank  Gian Pietro Pirola and  Alessandro Verra for useful conversations on the  subject and acknowledge funding
from MIUR Excellence Department Project CUP E83C180 00100006 (CC),
project FOSICAV within the  EU  Horizon
2020 research and innovation programme under the Marie
Sk{\l}odowska-Curie grant agreement n.~652782 (CC, ThD),
 GNSAGA of INDAM (CC, CG), Bergen Research
Foundation (ThD, ALK) and grant n.~261756 of the Research Council of Norway (ALK).

\section{Preliminary  results} \label{sec:proof}

\subsection{A basic lemma on Prym curves}\label{ssec:prym} 

The following is an immediate consequence of  the  Riemann-Roch theorem (see also \cite[\S 0.1]{cd} or \cite[Pf. of Lemma 2.1]{LS}). We include the proof for  the reader's convenience.

\begin{lemma} \label{lemma:Prym2}
  Let $(C,\eta)$ be any Prym curve of genus $g \geq 3$. Then:\\
  \begin{inparaenum}
  \item[(i)] $p$ is a base point of $|\omega_C(\eta)|$ if and only if $|p+\eta| \neq \emptyset$. This happens if and only if $C$ is hyperelliptic and $\eta \sim \O_C(p-q)$, with $p$ and $q$ ramification points of the $g^1_2$. In particular,
$p$ and $q$ are the only  base points;\\
\item[(ii)] if $|\omega_C(\eta)|$ is base point free, then it does not separate $p$ and $q$ (possibly infinitely near) if and only if $|p+q+\eta| \neq \emptyset$. This happens if and only if $C$ has a $g^1_4$ and $\eta \sim \O_C(p+q-x-y)$,
where $2(p+q)$ and $2(x+y)$ are members of the $g^1_4$. In particular, also
$x$ and $y$ are not separated by $|\omega_C(\eta)|$.
\end{inparaenum}
\end{lemma}

\begin{proof}
  We prove only (ii) and leave (i) to the reader. Assume that $|\omega_C(\eta)|$ is base point free. Then $p$ and $q$ are not separated by the linear system $|\omega_C(\eta)|$ if and only if 
$h^0(\omega_C(\eta)-p)=h^0(\omega_C(\eta)-p-q)$. By Riemann-Roch and Serre duality, this is equivalent to $h^0(\eta+p)+1=h^0(\eta+p+q)$. By (i), we have 
$h^0(\eta+p)=0$, whence the latter condition is $h^0(\eta+p+q)=1$.  This is equivalent to $h^0(\eta+p+q)>0$, because if $h^0(\eta+p+q)>1$, then we would have $h^0(\eta+p)>0$, a contradiction.  This proves the first assertion.

  We have $|p+q+\eta| \neq \emptyset$ if and only if $p+q+\eta \sim x+ y$, for  $x,y \in C$.
  This  implies $2(p+q) \sim 2(x+y)$, whence $C$ has a $g^1_4$ with 
$2(p+q)$ and $2(x+y)$ as its members. Conversely, if $2(p+q)$ and $2(x+y)$ are distinct members of a $g^1_4$ on $C$, then $\eta:=\O_C(p+q-x-y)$ is a $2$--torsion element of $\Pic^0(C)$ and satisfies the condition that  $|p+q+\eta| \neq \emptyset$. 
\end{proof}

The lemma says in particular that the locus in $\R_g$ of pairs $(C,\eta)$ for which the   Prym-canonical system $|\omega_C(\eta)|$ is not base-point free  
dominates the hyperelliptic locus via the forgetful map $\R_g \to \M_g$.

Recall that the tetragonal locus $\R^1_{g,4}$ is irreducible of dimension $2g+3$ if $g \geq 7$ and coincides with $\R_g$ if $g \leq 6$. Lemma \ref {lemma:Prym2} implies that $\R^0_g \subseteq \R^1_{g,4}$, thus proving the last statement in Theorem \ref{thm:main}. 

The lemma also enables us to detect the  locus  $\R^{0,\mathrm{nb}}_g$ in $\R^0_g$ where the Prym--canonical morphism is not birational onto its image:

\begin{corollary} \label{cor:notbir}
  Let $(C,\eta)$ be any Prym curve of genus $g \geq 4$ such that  the Prym-canonical system  $|\omega_C(\eta)|$ is base point free. If the Prym-canonical map is not birational onto its image, then it is of degree two onto a smooth elliptic curve. 

The locus  $\R^{0,\mathrm{nb}}_g$ is irreducible of dimension $2g-2$ and dominates the bielliptic locus in 
$\M_g$. More precisely,  $\R^{0,\mathrm{nb}}_g$  consists of pairs $(C,\eta)$, with $C$ bielliptic and $\eta:=\varphi^*\eta'$, where $\varphi :C \to E$ is a bielliptic map and 
$\eta'$ is a nontrivial $2$--torsion  element in $\Pic^0(E)$. 
\end{corollary}

\begin{proof}
   Let $(C,\eta)$ be as in the statement. Denote by $C'$ the image
of the Prym-canonical morphism $\varphi:C \to \PP^{g-2}$.  Let $\mu$ be the degree of $\varphi$ and $d$ the degree of $C'$. Then $d\mu=2g-2$ and, since $C'$ is non--degenerate in $\PP^{g-2}$, we must have $d \geq g-2$.  Since $g \geq 4$, then $2\leq \mu\leq 3$; moreover  $\mu=3$ implies  that  $g=4$ and $\varphi$ maps $C$ three-to-one to a conic. The latter case cannot happen:  indeed, we would have 
$\omega_C(\eta)=2\L$, where $|\L|$ is a $g^1_3$. Then $4\L= 2\omega_C$. Since $|2\omega_C|$ is cut out by quadrics on the canonical image of $C$ in $\PP^3$, it follows that the only quadric containing the canonical model is a cone. Then  $|\L|$ is the unique $g^1_3$ on $C$ and $2\L=\omega_C$, thus $\eta$ is trivial, a contradiction.
 
Hence $\mu=2$, and then $d=g-1$, so that $C'$ is a {\it curve of almost minimal degree}. It is easy to see, using the fact that $|\omega_C(\eta)|$ is complete, that $C'$ is a smooth elliptic curve (alternatively, apply \cite[Thm.~1.2]{bs}). Hence $C$ is bielliptic and any pair of points $p$ and $q$ identified by $\varphi$ satisfy
$p+q \sim \varphi^*(r)$ for a point  $r \in C'$. Thus $2p+2q \sim \varphi^*(2r)$ is a $g^1_4$. By Lemma \ref{lemma:Prym2}(ii) we have $\eta \sim \O_C(p+q-x-y)$, where also $\varphi(x) =\varphi(y)$, whence $x+y \sim \varphi^*(z)$, for a $z \in
C'$. Hence, again by Lemma \ref{lemma:Prym2}(ii), we have
$\eta \sim p+q-x-y \sim \varphi^*(r-z)$ and $r-z$ is a nontrivial  $2$--torsion element in $\Pic^0(C')$, because $\varphi^*:\Pic^0(C')\to \Pic^0(C)$ is injective.  

Conversely, if $C$ is a bielliptic curve, it admits at most finitely
many double covers $\varphi:C \to E$ onto an elliptic curve (cf.\ e.g.,
\cite{bd};   in fact,
for $g \geq 6$, it admits a unique such map), and
for any such $\varphi$ and any nontrivial $2$--torsion element $\eta'$
in $\Pic^0(E)$, we have $\eta' \sim r-z$, for $r, z \in E$. Letting
$\varphi^*(r)=p+q$ and $\varphi^*(z)=x+y$, we see that $2(p+q) \sim
2(x+y)$ and $\eta =\varphi^*\eta'$ satisfies the conditions of Lemma
\ref{lemma:Prym2}(ii). 

 We have therefore proved that $\R^{0,\mathrm{nb}}_g$  consists of pairs $(C,\eta)$, with $C$ bielliptic and $\eta:=\varphi^*\eta'$, where $\varphi:C \to E$ is a bielliptic map and 
$\eta'$ is a nontrivial $2$--torsion  element in $\Pic^0(E)$. 

 The statement about the dimension of $\R^{0,\mathrm{nb}}_g$ follows since the bielliptic locus has dimension $2g-2$. To prove its irreducibility, consider the  map $f: \R^{0,\mathrm{nb}}_g \to \R_1$ associating to $(C,\eta)$ the pair $(E,\eta')$ as above.  
 We study the fibres of this map. Consider the following obvious cartesian diagram defining $\H$, where $U\subset  \Sym^{2g-2}(E)$ is the open subset consisting of reduced divisors:
 \[
 \xymatrix{
 \mathcal H \ar[d] \ar[r] & \Pic^{g-1}(E) \ar[d]^{\otimes 2}\\
 U \ar[r] &\Pic^{2g-2}(E)
 }
\]
By Riemann's existence theorem, $\mathcal H/\Aut(E)$ is in one-to-one correspondence with the two-to-one covers of $E$ branched  at $2g-2$ points. Then the fibre of $f$  over $(E,\eta')$ is
isomorphic to $\mathcal H/\Aut(E)$ by what we said above. Now note that $\mathcal H$ is irreducible, since it fibres over (an open subset of) $\Pic^{g-1}(E)$ with fibres that are projective spaces of dimension $2g-3$. Hence also $\mathcal H/\Aut(E)$ is irreducible.

The irreducibility of $\R^{0,\mathrm{nb}}_g$ now follows from the 
irreducibility of $\R_1$. Actually  $\R_1$ is irreducible and rational.
To see this consider the irreducible family of elliptic curves 
$y^2=x(x-1)(x-\lambda)$, where $\lambda\in\mathbb C\setminus{\{0,1\}}$. 
The three non--trivial points of order two of the fibre $\mathcal 
C_\lambda$ over $\lambda$ may be identified with the points $(0,0), 
(1,0)$ and $(\lambda,0)$. Moreover, the $j$-invariant of the fibres
defines a six-to-one map $j:\mathbb C\setminus{\{0,1\}}\to \mathcal 
M_1$. Now consider on this family the
two sections defined by the points $(0,0), (1,0)$ which stay fixed as 
$\lambda$ varies. It is an exercise to prove that the irreducible family 
of two-marked elliptic curves we obtain in this way is isomorphic to the 
moduli space of pairs
$(C, (\eta_1,\eta_2))$ where $C$ is a smooth elliptic curve and 
$(\eta_1,\eta_2)$ is an ordered pair of  distinct non--trivial
$2$--torsion  
points of $\rm{Pic}^0(C)$. This moduli space is, in turn, isomorphic to 
the moduli space $\mathcal M_1^{(2)}$ of elliptic curves with a 
level $2$ structure \cite[Ex. 2.2.1]{hm}. Finally $\mathcal 
M_1^{(2)}\simeq \mathbb C\setminus{\{0,1\}}$
maps two-to-one dominantly to $\mathcal R_1$, 
via the map $(C, (\eta_1,\eta_2)) \mapsto (C, (\eta_1+\eta_2))$.
This proves the statement.
\end{proof}

\subsection{A result on linear systems on rational surfaces}\label{ssec:lin} We will need the following:

\begin{thm}[cf.\ {\cite[Cor.~(4.6)]{ac}}] 
\label{thm:ac}  Let $X$ be a smooth projective rational surface and $\delta$ a non--negative integer. Let $\L$ be a complete linear system on $X$ such that:\\
\begin{inparaenum}
\item [(i)] the general curve in $\L$ is smooth and irreducible;\\
\item [(ii)]   the genus $p_a(\L)$ of the general curve in $\L$ satisfies $p_a(\L)  \geq \delta$;\\
\item [(iii)] $\dim(\L)>3\delta$;\\
\item [(iv)] if $p_1,\ldots, p_\delta$ are general points of $X$, there is an element $C$ of $\L$ singular at $p_1,\ldots, p_\delta$ such that for each irreducible  component $C'$ of $C$ one has $K_X\cdot C'<0$. 
\end{inparaenum}

Then, if $p_1,\ldots, p_\delta$ are general points of $X$ and $\L(p^2_1,\ldots, p^2_\delta)$ is the subsystem of $\L$ formed by the curves singular at  $p_1,\ldots, p_\delta$, one has:\\
\begin{inparaenum}
\item [(a)] the general curve in $\L(p^2_1,\ldots, p^2_\delta)$ is irreducible, has nodes at $p_1,\ldots, p_\delta$ and no other singularity;\\
\item [(b)] $\dim(\L(p^2_1,\ldots, p^2_\delta))=\dim(\L)-3\delta$.
\end{inparaenum}
\end{thm}
\begin{proof} The proof of (a) is in \cite {ac}. As for (b), one has $\dim(\L(p^2_1,\ldots, p^2_\delta))=\dim(\L)-3\delta+\varepsilon$, with $\varepsilon\geq 0$. Consider the locally closed family of curves in $\L$ given by
\[
\F:=\bigcup_{p_1,\ldots, p_\delta} \L(p^2_1,\ldots, p_\delta^2),
\]
where the union is made by varying $p_1,\ldots, p_\delta$ among all the $\delta$--tuples   of sufficiently general points of $X$. Of course 
\[
\dim(\F)=2\delta+\dim(\L(p^2_1,\ldots, p^2_\delta))=\dim(\L)-\delta+\varepsilon.
\]
On the other hand, if $C$ is a general element in $\F$, it has nodes at
$p_1,\ldots, p_\delta$ and no other singularity by (a), hence the
Zariski tangent space to $\F$ at $C$ is the linear system
$\L(p_1,\ldots, p_\delta)$ of curves in $\L$ containing
$p_1,\ldots, p_\delta$. Since $p_1,\ldots, p_\delta$ are general, we
have $\dim(\L(p_1,\ldots, p_\delta))=\dim (\L)-\delta$, which proves
that $\varepsilon=0$. \end{proof}

\section{Irreducibility and unirationality of $\R^0_g$}\label{sec:irr}

In this section we prove a first part of Theorem \ref{thm:main}, namely:

\begin{proposition} \label{prop:Prym}
  The locus $\R^0_g$ is irreducible and unirational  for $g \geq 5$.
\end{proposition}

The proof is inspired by the arguments in \cite {ac} and requires some preliminary considerations. In \cite[Theorem (5.3)] {ac} the authors prove that some Hurwitz schemes
$\H_{g,d}$ are unirational. Here we focus on the case $d=4$ and recall their construction.

Fix $g=2h+\epsilon\geq 3$, with $0\leqslant \epsilon\leqslant 1$. Then set $n=h+3+\epsilon$ and 
\[
\delta={{n-1}\choose 2}-{{n-4}\choose 2}-g=h+2\epsilon.
\]
Fix now $p, p_1,\ldots, p_\delta$ general points in the projective plane and consider the linear system $\L_n(p^{n-4}, p^2_1,\ldots, p_\delta^2)$ of plane curves of degree $n$ having multiplicity at least $n-4$ at $p$ and multiplicity at least $2$ at $p_1,\ldots, p_\delta$. As an application of Theorem \ref {thm:ac}, in \cite[Cor. (4.7)]{ac} one proves that the dimension of $\L_n(p^{n-4}, p^2_1,\ldots, p_\delta^2)$  is the expected one, i.e., 
\[
\dim(\L_n(p^{n-4}, p^2_1,\ldots, p_\delta^2))={\frac {n(n+3)} 2}-{\frac {(n-4)(n-3)}2}-3\delta=2h+9-\epsilon,
\]
and the general curve $\Gamma$ in $\L_n(p^{n-4}, p^2_1,\ldots, p_\delta^2)$ is irreducible,  has an  ordinary  $(n-4)$--tuple point at $p$,
nodes at $p_1,\ldots, p_\delta$,  and no other singularity. The normalization $C$ of $\Gamma$ has genus $g$ and it has a $g^1_4$, which is the pull--back  to $C$ of the linear series cut out on $\Gamma$ by the pencil of lines through $p$. 

Consider then the locally closed family of curves
\[
\H:=\bigcup_{p_1,\ldots,p_{\delta}}  \L_n(p^{n-4}, p^2_1,\ldots, p_\delta^2),
\]
where the union is made by varying $p_1,\ldots, p_\delta$ among all the $\delta$--tuples  of sufficiently general points of the plane. Then $\H$ is clearly irreducible, rational, of dimension $\dim(\L_n(p^{n-4}, p^2_1,\ldots, p_\delta^2))+2\delta=4h+9+3\epsilon$, and in  \cite {ac} it is proved that the natural map
$\H\dasharrow \M_{g,4}^1$ is dominant, so that $\M_{g,4}^1$ is unirational.

\begin{proof} [Proof of Proposition \ref {prop:Prym}] To prove our
  result, we slightly modify the above argument from \cite {ac}. 
Let us fix $g\geq 5,n, \delta$ as above. 
Let $p, p_1,\ldots, p_\delta$ be general points in the plane.

\begin{claim}\label{cl:1} 
Consider the linear system $\L_{n-2}(p^{n-6}, p^2_1,\ldots,
p_\delta^2)$ of plane curves of degree $n-2$, having a point of
multiplicity at least $n-6$ at $p$, and singular at $p_1,\ldots,
p_\delta$. Then the dimension of $\L_{n-2}(p^{n-6}, p^2_1,\ldots,
p_\delta^2)$ is the expected one, i.e.,
\[
\dim(\L_{n-2}(p^{n-6}, p^2_1,\ldots, p_\delta^2))=\frac {(n-2)(n+1)}2 -\frac {(n-6)(n-5)}2-3\delta=2h-1-\epsilon.
\]
\end{claim}

\begin{proof}[Proof of  Claim \ref {cl:1}] Assume first $g=5$, which implies 
$(h,\epsilon,n,\delta)=(2,1,6,4)$. 
Then one has $\L_{n-2}(p^{n-6}, p^2_1,\ldots, p_\delta^2)=\L_4(p^2_1,\ldots, p_4^2)$, which consists of all pairs of conics through 
$p_1,\ldots, p_4$, and has dimension $2$ as desired.
We can assume next that $g\geq 6$, hence $h\geq 3$ and $n \geq 6$. 

Let $X$ be the blow--up of $\PP^2$ at $p$. Note that the anticanonical system of $X$ is very ample. Consider the linear system $\L$ proper transform on  $X$ of $\L_{n-2}(p^{n-6})$. One checks that $X$ and $\L$ verify the hypotheses (i)--(iv) of Theorem \ref {thm:ac}. Indeed, (i) and (iv) are immediate, whereas (ii) and (iii) follow by standard computations and the fact that $h\geq 3$. Then the assertion follows by Theorem \ref {thm:ac}(b).  \end{proof}

Next fix two distinct lines $r_1,r_2$ through $p$ and, for $1\leqslant i\leqslant 2$, two distinct points $q_{ij}$, both different from $p$, on the line $r_i$, with $1\leqslant j\leqslant 2$. Consider then the linear system $\L_n(p^{n-4}, p^2_1,\ldots, p_\delta^2; [q_{11},q_{12}, q_{21}, q_{22}])$ consisting of all curves in  $\L_n(p^{n-4}, p^2_1,\ldots, p_\delta^2)$ whose intersection multiplicity with $r_i$ at $q_{ij}$ is at least 2, for $1\leqslant i,j\leqslant 2$.

\begin{claim}\label{cl:2} The linear system $\L_n(p^{n-4}, p^2_1,\ldots, p_\delta^2; [q_{11},q_{12}, q_{21}, q_{22}])$ has the expected dimension, i.e.,
\begin{eqnarray*}
\dim(\L_n(p^{n-4}, p^2_1,\ldots, p_\delta^2; [q_{11},q_{12}, q_{21}, q_{22}])) & = &
\frac {n(n+3)}2-\frac {(n-4)(n-3)}2-3\delta-8 \\ & = & 2h+1-\epsilon,
\end{eqnarray*}
and the general curve in $\L_n(p^{n-4}, p^2_1,\ldots, p_\delta^2; [q_{11},q_{12}, q_{21}, q_{22}])$ is irreducible, has a point of multiplicity $n-4$ at $p$, has nodes at $p_1,\ldots, p_\delta$ and no other singularity, and is tangent  at $r_i$ in $q_{ij}$, for $1\leq i,j\leq 2$.
\end{claim}

\begin{proof}[Proof of  Claim  \ref {cl:2}] Let $X$ be the blow--up of the plane at $p$, at the points $q_{i,j}$ and at the infinitely near points to $q_{ij}$ along the line $r_i$, for $1\leqslant i,j\leqslant 2$. Note that the anticanonical system of $X$ has a fixed part consisting of the strict transforms $R_1, R_2$ of $r_1,r_2$ plus the exceptional divisor $E$ over $p$, and a movable part consisting of the pull back to $X$ of the linear system of the lines in the plane.

Let $\L$ be  the strict transform on $X$ of
$\L_n(p^{n-4};   [q_{11},q_{12}, q_{21}, q_{22}])$, the linear system of curves of degree $n$ with multiplicity at least $n-4$ at $p$ and whose intersection multiplicity with $r_i$ at $q_{ij}$ is at least 2, for $1\leqslant i,j\leqslant 2$. One has
\[
\dim(\L)={\frac {n(n+3)} 2}-{\frac {(n-4)(n-3)}2}-8
\]
and an application of Bertini's theorem shows that the general curve in $\L$ is smooth and irreducible and its genus is 
\[
p_a(\L)={{n-1}\choose 2}-{{n-4}\choose 2}\geq \delta.
\]
 Moreover 
\[
\dim(\L)-3\delta=2h+1-\epsilon>0
\]
Hence the linear system $\L(p_1^2,\ldots, p_\delta^2)$ of curves in $\L$ singular at $p_1,\ldots, p_\delta$ has dimension 
\[
\dim(\L(p_1^2,\ldots, p_\delta^2))\geq 2h+1-\epsilon.
\]
We claim that  $\L(p_1^2,\ldots, p_\delta^2)$ does not have $R_1, R_2$ or $E$ in its fixed locus. Indeed, if $E$ is in this fixed locus, then clearly also $R_1$ and $R_2$ split off $\L(p_1^2,\ldots, p_\delta^2)$. If $R_1$ is in the fixed locus, then by symmetry, also $R_2$ is in the fixed locus. So, suppose by contradiction that
$R_1, R_2$ are in the fixed locus.   
Then, after removing them from $\L(p_1^2,\ldots, p_\delta^2)$ we would remain with $\L'$, the pull--back to $X$ of  $\L_{n-2}(p^{n-6}, p_1^2,\ldots, p_\delta^2)$, which, by Claim \ref {cl:1}, has dimension $2h-1-\epsilon$. Hence we would have 
\[
2h-1-\epsilon=\dim(\L_{n-2}(p^{n-6}, p_1^2,\ldots, p_\delta^2))=\dim(\L(p_1^2,\ldots, p_\delta^2))\geq  2h+1-\epsilon,
\]
a contradiction. 

Let now $C$ be a general curve in $\L(p_1^2,\ldots, p_\delta^2)$. The
above argument implies that no component of $C$ is a fixed curve of
the anticanonical system of $X$. Then for any irreducible component
$C'$ of $C$ one has $K_X\cdot C'<0$. In conclusion, $\L$ verifies the
hypotheses (i)--(iv) of Theorem \ref {thm:ac}, and Claim~\ref{cl:2}
follows by the latter theorem.
\end{proof}

We now end the proof of Proposition \ref {prop:Prym}.
Consider the locally closed family of curves
\[
\G:=
 \bigcup_{p_1,\ldots, p_\delta,r_1,r_2,q_{11},q_{12}, q_{21}, q_{22}}  \L_n(p^{n-4}, p^2_1,\ldots, p_\delta^2; [q_{11},q_{12}, q_{21}, q_{22}])
\]
where the union is made by varying $p_1,\ldots, p_\delta$ among all $\delta$--tuples  of general  distinct  points of $X$, $r_1,r_2$ among all pairs of distinct lines through $p$ and $q_{ij}\not=p$ among all pairs of  distinct points of $r_i$, for $1\leq i,j\leq 2$.

Of course $\G$ is irreducible and rational, and we have a map $\alpha: \G\dasharrow \R^ 0_g$ which sends a general curve $\Gamma\in \G$ to $(C,\eta)$, where $C$ is the normalization of $\Gamma$, and $\eta=\O_C(q_{11}+q_{12}-q_{21}-q_{22})$, where, by abusing notation,  we denote by $q_{ij}$ their inverse images  in $C$, for $1\leq i,j\leq 2$. 
Since $\H\dasharrow \M_{g,4}^1$ is dominant by \cite[\S5] {ac}, then $\alpha$ is also dominant  by Lemma \ref{lemma:Prym2}.%
\footnote{\label{foot-add}
In fact this statement requires additional arguments. We provide those
in an Addendum (Section~\ref{addendum}).}
This proves the proposition.
\end{proof}

\section{Dimension of $\R^0_g$}\label{sec:dim}

In this section we finish the proof of Theorem \ref{thm:main} with the:

\begin{proposition} \label{prop:Prym3} 
   The irreducible locus $\R^0_g$ has dimension $2g+1$ 
if $g\geq 5$. 
\end{proposition}

\begin{proof}
Let
$\H_{g,4}$ denote the  Hurwitz scheme parametrizing  isomorphism classes of  genus $g$ degree $4$ covers of $\PP^1$. We have a commutative diagram 
\[
\xymatrix{\G \ar@{-->}[d]_{\alpha} \ar@{-->}[r]^{\hspace{-0.2cm}\varphi} & \H_{g,4} \ar[d]^{\pi} \\
\R^0_{g} \ar[r]^\psi & \M^1_{g,4},
}
\]
where $\pi$ and $\psi$ are the forgetful maps, $\alpha$ is the dominant map from the last part of the proof of Proposition \ref{prop:Prym} and $\varphi$ maps a general curve $\Gamma \in \G$ to the  degree 4  cover defined by $2(q_{11}+q_{12}) \sim 2(q_{21}+q_{22})$, using the notation of the proof of Proposition \ref{prop:Prym}.  Note that  $\psi$ is  finite,  whence the dimension of $\R^0_g$ equals the dimension of the image of $\pi \circ \varphi$.

The image of $\varphi$ coincides with the locus  $\D \subset \H_{g,4}$
parametrizing covers with two pairs of distinct ramification points
 each over the same \ branch point. By
Riemann's existence theorem, $\D$ has codimension $2$ in $\H_{g,4}$
(whence $\dim( \D)=2g+1$).  Since  $\G$  is
irreducible  (cf.\ the proof of Proposition \ref
{prop:Prym}),  so is $\D$. Moreover, as the bielliptic
locus in $\M_g$ has dimension $2g-2$ and  each bielliptic curve has a
one-dimensional family of $g^1_4$s, the locus in $\H_{g,4}$ with
bielliptic  domain  curve has dimension $2g-1$. Thus, the general
element in the image of $\pi \circ \varphi$ is not bielliptic, whence
the general element  $(C,\eta)$  in the image of $\alpha$ 
 has Prym--canonical image birational to $C$,
 by Corollary \ref{cor:notbir} (and necessarily singular, by Lemma \ref{lemma:Prym2}(ii)). It follows that the fibre over $C$ of the 
restriction of $\pi$ to $\D$  is  finite. Indeed, $C$ has finitely many preimages $(C,\eta)$ in the image of $\alpha$, and the Prym--canonical model of each of those has finitely many singular points, determining by Lemma \ref{lemma:Prym2}(ii) only finitely many covers in $\D \subset \H_{g,4}$ mapping to $C$ by $\pi$. Thus, the 
restriction of $\pi$ to $\D$  is generically  finite, whence the image of $\pi \circ \varphi$ has dimension $2g+1$. 
\end{proof}

\section{Proof of proposition \ref{prop:nod} and some examples} \label{sec:prpr}

Consider again the locus  $\D \subset \H_{g,4}$ from the proof of Proposition 
\ref{prop:Prym3} 
parametrizing  isomorphism classes of  covers
with two pairs of distinct ramification points 
each over a single branch point. By Riemann's
existence theorem again, the general point in $\D$ corresponds to a
cover with only two such branch points. By Lemma
\ref{lemma:Prym2}(ii), if the  domain  curve
has only one $g^1_4$, which is automatic if $g \geq 10$, then the
Prym--canonical model of such a curve has precisely two nodes. It
cannot have fewer singularities by Lemma \ref{lemma:Prym2}. 
Thus, Proposition \ref{prop:nod} is proved for $g \geq 10$. 

Instead of embarking in a  more refined treatment for $g \leq 9$, we note that certain curves on Enriques surfaces provide examples, for any genus $g \geq 5$, of curves with two--nodal Prym--canonical models, thus finishing the proof of Proposition \ref{prop:nod}:

\begin{example} \label{exa:enr1}
  The general Enriques surface $S$ contains no smooth rational curves \cite{bp} and contains smooth elliptic curves $E_1,E_2,E_3$ with $E_i \cdot E_j=1$ for $i \neq j$ (and $E_i^2=0$ by adjunction),  for $1\leq i, j\leq 3$,  cf.\ e.g. \cite[Thm. 3.2]{cos2} or \cite[IV.9.E, p.~273]{cd}. It also contains a smooth elliptic curve $E_{1,2}$ such that $E_{1,2} \cdot E_1=E_{1,2} \cdot E_2=2$, and $E_{1,2} \cdot E_3 =1$, cf.\ e.g. \cite[Thm. 3.2]{cos2} or \cite[IV.9.B, p.~270]{cd}. In particular, none of the numerical equivalence classes of $E_1,E_2,E_3,E_{1,2}$ are divisible in $\Num(S)$.

Consider, for any $g \geq 5$, the line bundle
\[  H_g := \begin{cases} \O_S(\frac{g-2}{2}E_1 +E_2 +E_3 ), & g \; \; \mbox{even} \\
 \O_S(\frac{g-1}{2}E_1 +E_{1,2} ), & g \; \; \mbox{odd}.\end{cases}
\]

The absence of smooth rational curves yields that $H_g$ is nef. As  $H_g^2=2g-2$,  all curves in $|H_g|$ have
arithmetic genus $g$. Moreover, we claim that $\phi(H_g)=E_1 \cdot H_g=2$ (see the introduction for the definition of $\phi$) and that the only numerical class computing $\phi(H_g)$ is $E_1$.
Indeed, if $g$ is even (respectively, odd), then $E_1 \cdot H_g=2$, $E_2 \cdot H_g=E_3 \cdot H_g=\frac{g}{2} \geq 3$ (resp., $E_1 \cdot H_g=2$, $E_{1,2} \cdot H_g=g-1 \geq 4$), and if $E$ is any nonzero effective divisor not numerically equivalent to any of $E_1,E_2,E_3$ (resp., $E_1,E_{1,2}$), then $E \cdot E_1>0$, $E \cdot E_2>0$ and $E \cdot E_ 3>0$ (resp., $E \cdot E_1>0$ and $E \cdot E_{1,2}>0$) by \cite[Lemma 2.1]{klvan}, so that $E \cdot H_g \geq \frac{g-2}{2}+2=\frac{g}{2}+1 \geq 4$ (resp.,  $E \cdot H_g \geq \frac{g-1}{2}+1=\frac{g+1}{2} \geq 3$).

By \cite[Prop. 4.5.1, Thm. 4.6.3, Prop. 4.7.1, Thm. 4.7.1]{cd}  the complete linear system $|H_g|$ is  therefore base point free and defines a 
morphism $\varphi_{H_g}$ that is birational onto a surface with only
double lines as singularities; the double lines are the images of
curves computing $\phi(H_g)$, which, by what we said above, are $E_1$ and $E'_1$,
the only member of $|E_1+K_S|$.  Thus, the image of $\varphi_{H_g}$ is
a surface with precisely two double lines $\varphi_{H_g}(E_1)$ and
$\varphi_{H_g}(E_1')$ as singularities. Therefore, $\varphi_{H_g}$ maps a general smooth $C \in |H|$  to a curve with precisely two nodes. Since  $\varphi_{H_g}$ restricted to $C$ is the Prym--canonical map associated to $\eta:=\O_C(K_S)$ by \cite[Cor. 4.1.2]{cd}, a general smooth curve $C$ in $|H_g|$ together with $\eta$ is an example of a Prym curve of any genus $g \geq 5$ with two--nodal Prym--canonical model. 

 We prove in  \cite[Thm. 2]{cdgk} that the general element in $\R^0_g$ is obtained in this way precisely for $5 \leq g \leq 8$.  

Similar examples for odd $g \geq 7$ are obtained from the line bundle 
$H_g:=\O_S(\frac{g-1}{2}E_1 +2E_{2})$  or $H_g:=\O_S(\frac{g-1}{2}E_1 +2E_{2}+K_S)$,  but (again by \cite[Thm. 2]{cdgk}) the general element in $\R^0_g$ is not obtained in this way. 
\end{example}

We conclude with an example of curves of genus $5$ on an Enriques surface with $4$-nodal Prym--canonical models and a result that will be used in \cite{cdgk}:

\begin{example} \label{exa:enr2}
  With the same notation as in the previous example, set $H:=\O_S(2E_1+2E_2+K_S)$.  Then $H^2=8$, so that any curve in $|H|$ has arithmetic genus $5$. Moreover, $\phi(H)=2$ and one easily checks that $E_1$ and $E_2$ are the only numerical equivalence classes computing $\phi(H)$. As in the previous example, the complete linear system  $|H|$  is base point free and defines a 
morphism  $\varphi_{H}$  that is birational onto a surface with precisely four double lines as singularities, namely the images of $E_1$, $E_2$, $E_1'$ and $E'_2$, where $E'_i$ is the only member of $|E_i+K_S|$, $i=1,2$. Thus $\varphi_H$ maps a general smooth $C \in |H|$  to a curve with precisely four nodes, so that, again by \cite[Cor. 4.1.2]{cd}, the pairs $(C,\O_C(K_S))$ are genus $5$ Prym curves with $4$-nodal Prym--canonical models. 

Also note that for any smooth $C \in |H|$, we have
\[ \omega_C \cong \O_C(E_1+E_2)^{\* 2} \cong \O_C(E_1+E_2+K_S)^{\* 2},\]
whence $C$ has two autoresidual $g^1_4$s, namely $|\O_C(E_1+E_2)|$ and
$|\O_C(E_1+E_2+K_S)|$,  and their difference is $\O_C(K_S)$.  (A complete linear system $|D|$  is called {\it autoresidual} if $D$ is a theta-characteristic, that is, $2D \sim \omega_D$.) 
Thus,  $(C,\O_C(K_S))$ 
 belongs to the  locus  in  $\R_5$  consisting of  Prym
  curves  $(C,\eta)$  carrying a theta-characteristic $\theta$ such that
$h^0(\theta)=h^0(\theta+\eta)=2$.  The next result shows that
this is a general phenomenon  in $\R^0_5$. 
\end{example}

\begin{proposition}
  The  locus  in  $\R^0_5$  of curves with $4$-nodal Prym-canonical model is an irreducible  unirational  divisor  whose closure in $\R_5$
coincides with the closure of the locus of Prym curves $(C,\eta)$ carrying a theta-characteristic $\theta$
with
$h^0(\theta)=h^0(\theta+\eta)=2$. 
\end{proposition}

\begin{proof}
  Let us denote by $\D^0_5$ the locus of curves in $\R^0_5$ with $4$-nodal Prym-canonical model, which is  nonempty by the previous example.
Let $\V$ denote the locus of curves of type $(4,4)$ on $\PP^1 \x \PP^1$ with $4$ nodes lying on the $4$ nodes of a ``square'' configuration of two fibres of each projection to $\PP^1$. We will prove that $\V$ is irreducible of dimension $16$ and that there is a birational  morphism 

\[ f: \D^0_5 \longrightarrow \V':=\V/\Aut (\PP^1 \x \PP^1). \]

To define $f$, let $(C,\eta) \in \D^0_5$. By Lemma \ref{lemma:Prym2} there are four pairs of distinct points $(p,q)$, $(x,y)$, $(p',q')$ and $(x',y')$ on $C$, each identified by the Prym--canonical map $\varphi:C \to \PP^3$, such that
\begin{eqnarray} 
\label{eq:g14}
  2(p+q) \sim 2(x+y), \; \; 2(p'+q') \sim 2(x'+y') \; \; \mbox{and} \\
\label{eq:eta} \eta \sim p+q-x-y \sim x'+y'-p'-q'.  
\end{eqnarray}
In particular, we get that
\begin{equation}
  \label{eq:ell1}
 p+q+p'+q' \sim  x+y+x'+y',
\end{equation}
thus defining a base point free $g^1_4$ on $C$, which we call $\ell_1$. We let $\L_1$ on $C$ be the corresponding line bundle.  Since there exists a pencil of hyperplanes in $\PP^3$ through any two of the four nodes of $\Gamma:=\varphi(C)$, we see that 
\begin{equation} \label{eq:meno1}
h^0(\omega_C(\eta)-\L_1)=h^0\bigl(\omega_C(\eta)(-p-q-p'-q')\bigr)=2.
\end{equation}
We claim that
\begin{equation} \label{eq:meno2}
h^0(\omega_C(\eta)-2\L_1)=0.
\end{equation}
Indeed, if not, we would have $\omega_C(\eta) \cong 2\L_1$, which together with \eqref{eq:meno1} would yield that  $\Gamma \subset \PP^3$  is contained in a quadric cone $Q$,  
with the pullback of the ruling of the cone cutting $\ell_1$ on $C$. Let $\widetilde{Q}$ be the desingularization of $Q$. Then $\widetilde{Q} \cong \FF_2$. Since $\ell_1$ is base point free,   $\Gamma$  does not pass through the vertex of $Q$, so that we may consider   $\Gamma$  as a curve in $\widetilde{Q}$. Denote by $\sigma$ the minimal section of $\FF_2$ (thus, $\sigma^2=-2$), which is contracted to the vertex of $Q$, and by $\f$ the class of the fibre  of the ruling.  Then, since  $\Gamma \cdot \f=4$  and  $\Gamma \cdot \sigma=0$, we get that   $\Gamma \sim 4 \sigma+8 \f$.  In particular,  $\omega_{\Gamma} \cong \O_{\Gamma}(K_{\widetilde{Q}}+\Gamma)) \cong
\O_{\Gamma}(2\sigma+4\f) \cong \O_{\Gamma}(4\f)$.  Thus, from \eqref{eq:ell1}
we obtain 
\[ \omega_C \cong \varphi^*(\omega_{\Gamma}) (-p-q-x-y-p'-q'-x'-y') \cong \O_C(4\L_1-2\L_1) \cong \O_C(2\L_1),\]
yielding $\eta=0$, a contradiction.
This proves \eqref{eq:meno2}.

The relations \eqref{eq:meno1} and \eqref{eq:meno2} imply that   $\Gamma \subset \PP^3$  is contained in a smooth quadric surface $Q \cong \PP^1 \x \PP^1$. 
The first ruling is defined by the pencil $\ell_1$, whereas the second is defined by the pencil $\ell_2=|\L_2|$, where $\L_2:=\omega_C(\eta)-\L_1=\omega_C(\eta)(-p-q-p'-q')$ by \eqref{eq:meno1}. The curve 
 $\Gamma$  is of type $(4,4)$ on $Q$, with four nodes. Since  $\omega_{\Gamma} \cong \omega_{\PP^1 \x \PP^1}(C) \cong \O_{\Gamma}(2,2)$, we see that
 $\varphi^*(\omega_{\Gamma}) \cong  (\omega_C(\eta))^{\*2} \cong \omega_C^{\*2}$. Thus, 
\[ \omega_C \cong  \omega_C^{\*2}(-p-q-x-y-p'-q'-x'-y'), \]
whence
\begin{equation}
  \label{eq:can}
  \omega_C \cong \O_C(p+q+x+y+p'+q'+x'+y').
\end{equation}
Combining with \eqref{eq:eta}, we find that 
\begin{equation}
  \label{eq:ell2} \L_2 \cong \omega_C(\eta)(-p-q-p'-q') \cong \O_C(p+q+x'+y') \cong \O_C(p'+q'+x+y).
\end{equation}
The relations \eqref{eq:ell1} and \eqref{eq:ell2} tell us that the four nodes of  $\Gamma$  lie on two pairs of fibres of each ruling of $\PP^1 \x \PP^1$, thus showing that  $\Gamma \in \V$. Of course this is all well-defined up to automorphisms of $\PP^1 \x \PP^1$, so we see that the construction associates to $(C,\eta)$ an element in $\V'$, which we define to be the image of $(C,\eta)$ by $f$. 

This defines the map $f$, and in particular shows that $\V$ is nonempty. 
We also note for later use that $\omega_C \cong 2\L_1 \cong 2\L_2$, so that 
 $\D^0_5$ is contained in the locus of Prym curves $(C,\eta)$ carrying a theta-characteristic $\theta$ with
$h^0(\theta)=h^0(\theta+\eta)=2$, which we henceforth call $T_5$. Moreover, via the forgetful map $\R_5 \to \M_5$, the locus $T_5$ maps to the locus of curves with two (complete) autoresidual $g^1_4$s, which we call $\B_5$.

We next prove that $\V$ is irreducible  rational  of dimension $16$. 

For any $X \in \V$, let $\nu:C \to X$ be the normalization; $C$ has
genus $5$. If $z_i$, $i=1,2,3,4$, are the nodes of $X$, then the
complete linear system $|\O_{\PP^1\times\PP^1}(X) \* \I_{z_1}^2 \* \I_{z_2}^2 \* \I_{z_3}^2
\* \I_{z_4}^2|$ has dimension $12$, as expected. Indeed, letting $r$
be its dimension, we clearly have $r \geq 12$; on the other hand, this
complete linear system induces a $g^{r-1}_{16}$ on $C$, whence $r-1
\leq 11$ by Riemann-Roch. It follows that $\V$ is birational to
$\PP^{12} \x (\Sym^2(\PP^1))^2$ (because of the freedom of varying the
four lines in the square configuration), in particular it is
irreducible rational of dimension $12+4=16$.

We now define the inverse of $f$. Given a curve $X \in \V$, let $\L_1$ and $\L_2$ be the line bundles of degree $4$ on $C$ defined by the pullbacks of the two rulings on $\PP^1 \x \PP^1$. By the special position of the $4$ nodes of $X$, the four pairs of points $C$ lying above the four nodes of $X$, say $(p,q)$, $(x,y)$, $(p',q')$ and $(x',y')$, satisfy
\begin{eqnarray*}
  \L_1 & \cong & \O_C(p+q+p'+q')  \cong \O_C(x+y+x'+y'), \\
\L_2 & \cong & \O_C(p+q+x'+y')  \cong \O_C(x+y+p'+q'),
\end{eqnarray*}
in particular, $\eta:=\L_1-\L_2$ is $2$--torsion. Moreover, one can easily verify that  $\omega_C(\eta) \cong \L_1+\L_2$. Thus, the normalization $\nu: C \to X \subset \PP^1 \x \PP^1$ followed by the embedding of $\PP^1 \x \PP^1$ as a quadric in $\PP^3$ induces the Prym--canonical map associated to $\omega_C(\eta)$, so that $(C,\eta)$ has a $4$--nodal Prym--canonical image. One readily checks that this map is  the  inverse of the map $f$ defined above. Thus, we have proved that $\D^0_5$ is irreducible of dimension $\dim \V/(\Aut (\PP^1 \x \PP^1))= 16-6=10$. 

 We have left to prove that the closure of $\D^0_5$ in $\R_5$ coincides with the closure of  $T_5$. 
We proved above that $\D^0_5$ is contained in $T_5$ and that the latter maps, via the finite forgetful map $\R_5 \to \M_5$, to the locus $\B_5$ of curves with two autoresidual $g^1_4$s, which is irreducible of dimension $10$ by \cite[Thm. 2.10]{KLV}. Below we give a direct proof of the latter fact, which also proves that the general member of $\B_5$ carries exactly two $g^1_4$s, equivalently two theta characteristics
$\theta$ and $\theta'$ such that $h^0(\theta)=h^0(\theta')=2$. It will follow that there is an inverse rational map $\B_5 \dashrightarrow T_5$ mapping 
$C$ to $(C,\theta-\theta')$, proving that also $T_5$ is irreducible of dimension $10$. Its closure must therefore coincide with the closure of $\D^0_5$, finishing the proof  of the proposition. 

 So let $C$ be a smooth, irreducible curve of genus 5 and consider its canonical embedding $C \subset \PP^4$. Given  $\xi=|D|$ a (complete) $g^1_4$ on $C$, the divisors in $\xi$ span planes which sweep out a quadric $Q_\xi$ of rank $r<5$. If $\xi$ is not autoresidual, then $Q_\xi$ has rank $r=4$ and it has another 1--dimensional system of planes which cut out on $C$ the divisors of $\xi'=|K_C-D|$. In this case $Q_\xi=Q_{\xi'}$.
Hence $\xi$ is autoresidual if and only if $Q_\xi$ has rank 3, and therefore it possesses only one 1--dimensional family of planes.
This means that the homogeneous ideal of a curve in $\B_5$ in its canonical embedding contains two distinct rank 3 quadrics. Hence the general curve $C$ in $\B_5$ is obtained by intersecting two general rank 3 quadrics in $\PP^4$ with another general quadric. Note that the two rank 3 quadrics cut out a Del Pezzo surface $S$ with 4 nodes, hence $C$ is a general quadric section on $S$. The two autoresidual $g^1_4$ on $C$ are cut out on $C$ by the conics of the two pencils on $S$ with base points two of the nodes.

From this description it follows that $\B_5$ is irreducible, 10--dimensional and that its general member contains precisely two autoresidual $g^1_4$s. Indeed, consider the $\PP^{14}$ parametrizing all quadrics in $\PP^4$. The locus $\mathcal X$ of quadrics of rank $r\leq 3$ is non--degenerate and has dimension 11. The net of quadrics defining a general curve $C$ in $\B_5$ corresponds to a plane in $\PP^{14}$ containing a general secant line to $\mathcal X$ (which, by its generality, contains only two points in $\mathcal X$), and an easy count of parameters shows that these planes clearly fill up a variety of dimension 34. Modding out by the 24--dimensional group of projective transformations of $\PP^4$, we get dimension 10 for $\B_5$. 
\end{proof}

\begin{remark} \label{rem:beau}
Denote, as in the last proof, by $\D_0^5$ the locus of Prym curves $(C,\eta)$ carrying a theta-characteristic $\theta$
with
$h^0(\theta)=h^0(\theta+\eta)=2$. By \cite[Prop. 7.3 and Thm. 7.4]{be} the  locus $\D_0^5$  maps, via the Prym map $\P_5: \R_5 \to \A_4$, to the irreducible divisor
$\theta_{\tiny{\mbox{null}}}$ of principally polarized abelian varieties whose theta-divisor has a singular point at a $2$--torsion point, and moreover the general member of $\P_5(\D_0^5)$  has precisely one ordinary double point, cf. \cite[Pf.~of~Prop.~7.5]{be}. It would be 
interesting to know if $\D_0^5$ dominates $\theta_{\tiny{\mbox{null}}}$. 

By \cite[Prop. 7.3]{be} one knows that the closure of $\P_5^{-1}(\theta_{\tiny{\mbox{null}}})$ is the closure of the locus of Prym curves $(C,\eta)$ carrying a theta-characteristic $\theta$
such that
$h^0(\theta)+h^0(\theta+\eta)$ is even, which {\it properly} contains $\D_0^5$. 
\end{remark}

\begin{remark} \label{rem:enr}
  By contrast, if we  consider the adjoint line bundle of the one   in  Example \ref{exa:enr2},  that is,  $H':=\O_S(2E_1+2E_2)$,  then  by
\cite[Prop. 4.1.2, Thm. 4.7.1, (F) p.~277]{cd} the morphism $\varphi_{H'}$ defined by $|H'|$  is of degree $2$ onto a quartic Del Pezzo surface. In particular, $\varphi_{H'}$  maps any smooth $C \in |H|$ doubly onto  an elliptic quartic curve in $\PP^3$. Hence, the Prym curve $(C,\O_C(K_S))$ belongs to the locus $\R^{0,\mathrm{nb}}_5$  described in Corollary \ref{cor:notbir}.  
\end{remark}

%%%%%%%%%%%%%%%%%%%%%%%%%%%%%(BIBLIOGRAPHY)%%%%%%%%%%%%%%%%%%%%%%%%%%%%%%%
%
%
%%%%%%%%%%%%%%%%%%%%%%%%%%%%%%%%%%%%%%%%%%%%%%%%%%%%%%%%%%%%%%%%%%%%%%%

\clearpage

\def\loccit{\cite{loccit}}
\let\regbar\bar
\let\regtilde\tilde
\def\bar#1{\overline{#1}}
\def\tilde#1{\widetilde{#1}}

\null \vspace{0mm}

\renewcommand{\thesection}{\Alph{section}}
\setcounter{section}{0}
\refstepcounter{section}
\label{addendum}
\setcounter{section}{0}

{\Large\bfseries
\section{Addendum}}

\begin{abstract}
  We give some additional details on the proof of
\cite[Prop.~3.1]{loccit}, which states the irreducibiliy and
unirationality of the locus $\R^0_g$ of Prym curves such that the
Prym-canonical linear system is base point free but does not define an
embedding, for $g\geq 5$.
 \end{abstract}

\vspace{1cm}

We use freely the notation and setup introduced in \loccit. At the
very end of 
the proof of Prop.~3.1, p.~78
(see footnote~\ref{foot-add}), we claim that ``since $\H
\dashrightarrow \M_{g,4}^1$ is dominant, then $\alpha: \G
\dashrightarrow \R_g^0$ is also
dominant''.
The aim of this addendum is to provide a complete proof of this claim.

Let us denote by $\beta$ the map $\H
\dashrightarrow \M_{g,4}^1$ and by $U$ its domain of definition. Note that by \cite[Lemma~2.1]{loccit}, $\R_g^0$  is  a sublocus of 
  $\R_{g,4}^1$.  In order to prove the above mentioned claim, it is
 necessary and sufficient to check that there are no irreducible
components of $\R_g^0$  dominating a sublocus of $\M_{g,4}^1\setminus \beta(U)$. 
This had not been carried out in \loccit, and we provide such a
verification in the present addendum.

\medskip
 We denote by $\bar \M_g$ the moduli space of Deligne and Mumford stable curves of genus $g$.  The map $\beta$ extends to a rational map $\bar \beta: \bar{ \H} \dasharrow \bar\M\vphantom{M}^1_{g,4}$, where $\bar \H$ and $\bar \M\vphantom{M}^1_{g,4}$ are the closures of $\H$ and
$\M_{g,4}^1$ in $|\O_{\PP^2}(n)|$ and $\bar \M_g$ respectively. We denote by $\bar{\mathcal G}$ the closure of $\mathcal G$ in $\bar {\mathcal H}$.

\begin{claim}\label{cl:1} Let $\G'$ be an irreducible family of irreducible plane curves of
degree $n$ and genus $g$, having multiplicity $n-4$ at $p$,    whose
normalization, endowed with the $g^1_4$ cut out by
the lines through $p$, belongs to $\R_g^0 \subset \R_{g,4}^1$ . Then $\G'$ is contained in $\bar{\mathcal G}$.\end{claim}

\begin{proof} We may assume that $\G'$ is maximal for the above properties with respect to inclusion.   Let $C$ be a general member of $\G'$. 
By \cite[Lemma~2.1]{loccit} and the deformation theory for logarithmic
Severi varieties, see \cite[Thm.~1.4]{noteTD},
$C$ has an ordinary $(n-4)$-tuple point at
$p$ and its other singularities are all nodal.
Therefore $C$ is a member of some linear system 
$\L_n \bigl(p^{n-4},p_1^2,\ldots,p_\delta^2,
[q_1,q_2,q'_1,q'_2] \bigr)$, with the possibility that the points
$p_1,\ldots,p_\delta,q_1,q_2,q'_1,q'_2$ be in a special position.

However the dimension of these linear systems
is always the same, no matter the position of the points
$p_1,\ldots,p_\delta,q_1,q_2,q'_1,q'_2$: indeed the corresponding
characteristic series on $C$ is always non-special, as follows from
the computation of its degree, which does not depend on the position
of the points $p_1,\ldots,p_\delta,q_1,q_2,q'_1,q'_2$.
The upshot is that the union of the linear systems 
$\L_n \bigl(p^{n-4},p_1^2,\ldots,p_\delta^2,
[q_1,q_2,q'_1,q'_2] \bigr)$
is irreducible, even if one includes special configurations of points
$p_1,\ldots,p_\delta,q_1,q_2,q'_1,q'_2$ in taking the union.
This union therefore has $\G$ as
a dense subset, hence $C$ belongs to the closure of $\G$,
and the closures of $\G$ and $\G'$ are equal, which proves our claim.
\end{proof}

\begin{claim}\label{cl:2}
There is no
irreducible component 
of $\R_g^0$   dominating a sublocus of $\M_{g,4}^1\setminus \beta(U)$.
\end{claim}

\begin{proof}
We consider 
 $\tilde \H$ a blow-up of
$\bar \H$ such that the map $\beta: \H \dashrightarrow \M^1_{g,4}$
induces a regular map 
$\tilde \beta: \tilde \H \to \bar\M\vphantom{M}^1_{g,4}$
(in particular $\tilde\H \to \bar\H$ factorizes through the
normalization of $\bar\H$).
A point $\tilde C$ of the exceptional locus
of $\tilde\H \to \bar\H$ is the datum of a point $[C] \in \bar\H$
and
an infinitesimal curvilinear arc in $\bar\H$ centered at $[C]$,
and the image of $\tilde C$ in $\bar\M\vphantom{M}^1_{g,4}$
is determined by the stable reduction of the deformation of the
partial normalization of genus $g$ of $C$ corresponding to this curvilinear arc.
Since $\tilde \H$ and $\bar\M\vphantom{M}^1_{g,4}$ are projective and $\beta$ is dominant, 
$\tilde\beta$ is surjective; in particular the closure of  the image of  $\R_g^0$ in $\bar\M_g$ is contained in the image of $\tilde\beta$.
To prove our claim we may thus consider a point
$\tilde C$ in $\tilde{\mathcal H}$ sitting over $[C] \in \bar
{\mathcal H}$, such that $\tilde \beta(\tilde C)$ is
a general point in some irreducible component 
of the closure of  the image of  
$\R_g^0$ in $\bar\M_g$, and show that then $\tilde \beta(\tilde C)$
necessarily belongs to 
the closure of $\beta(\mathcal G)=\alpha(\G)$.

First consider the case that $C$ is irreducible and reduced.
If $C$ has geometric genus $g$, then $[C]$ sits in the domain of
definition of  $ \beta$ in $\bar\H$ (or rather of the rational map $\bar\beta$\, induced by
 $\beta$  on the normalization of $\bar\H$), and, by Claim \ref
{cl:1}, the curve $C$ 
is in the closure of $\mathcal G$. Then $\tilde \beta(\tilde C)=\bar
\beta([C])$ sits in the closure of $\beta(\mathcal G)=\alpha(\G)$, and
we are done. 

If the geometric genus of $C$ is smaller than $g$, then for all
$\tilde C$ in $\tilde{\mathcal H}$ over $[C]$,  the curve $\tilde
\beta(\tilde C)$ has an irreducible component of genus smaller than
$g$, so it sits in the boundary of  $\bar\M_g$, hence it cannot be a
general member of a component of the closure of  the image of  $\mathcal R^0_g$. This ends the proof in the case $C$ integral.

Assume next that $C$ is not integral. Let us first consider the case
in which $C$ has several irreducible components. If none of these is
contracted in the stable reduction, then the stable model of $\tilde
C$ is in the boundary of  $\bar\M_g$ , hence it cannot be a general
member of a component of the closure of  the image of  $\mathcal R^0_g$. 

Next we examine the case in which 
some component $C'$ is contracted in the stable reduction. 
If the restriction to $C'$ of the $g^1_4$ cut out by lines through $p$
is non-trivial, then the $g^1_4$ on the stable model of $\tilde C$ has
base points. If there is only one such base point, then the stable
model cannot lie in $\mathcal R^0_g$ by  \cite[Lemma~2.1]{loccit}. If
there are more base points, then the stable model is hyperelliptic,
hence the number of moduli of such curves is at most $2g-1$. On the other hand the irreducible components of $\R_g^0$ have dimension at least
$2g+1$ as a direct application of  \cite[Lemma~2.1]{loccit} shows.
Therefore in this case, $\tilde\beta(\tilde C)$ cannot be general in a
component of the closure of  the image of  $\mathcal R^0_g$. 

If the restriction to $C'$ of the $g^1_4$ cut out by lines through $p$ is trivial, then $C'$ is a line passing through $p$. In this case a direct computation shows that the number of moduli for $C-C'$ is at most $2g$, hence we conclude as above. 

So the only remaining case is the one in which $C=mC_0$
with $m>1$ and $C_0$ irreducible. We want to check that these curves cannot give rise to an irreducible component
of $\R_g^0$.
First, since the pencil of lines through $p$ cuts out a $g^1_4$ on the
general member of $\H$, one sees that $m$ may only take the values $2$
and $4$.
By explicit computation of the stable reduction,
as for instance in \cite[3.C]{HM},
one sees that
in the former case it gives a curve $C'$  of genus $g$ which is a
double cover of a hyperelliptic curve, and in the latter case it gives
a tetragonal curve for which all ramification points of the $g^1_4$
are triple. 

Suppose first $C'$ is a double cover of a hyperelliptic curve $\Gamma$ of genus $\gamma$. An easy count of parameters shows that the number of moduli on which $C'$ depends is at most $2g+3-2\gamma$ if $\gamma \geq 2$, is at most $2g-2$ if $\gamma=1$, and is at most $2g-1$ if $\gamma=0$. Suppose next that $C'$ has a $g^1_4$ for which all ramification points 
are triple. Then the number of moduli of $C'$ is at most $\frac 23g-1$. 

In any case the curves $C'$ have too few moduli to fill up a
component of $\R_g^0$. This proves our claim.
\end{proof}

\medskip
Claim~\ref{cl:2} above implies that indeed $\alpha: \G \dashrightarrow
\R_g^0$ is dominant. The proof of \cite[Prop.~3.1]{loccit} is thus
complete, and $\R_g^0$ is indeed irreducible and unirational.

\end{document}